\documentclass[reqno]{amsart}

\usepackage{amsfonts, amsmath, amsthm, graphicx, mathtools, comment, dsfont}

\usepackage{hyperref}

\usepackage{todonotes}

\usepackage{orcidlink}

\newcommand{\dd}{\mathrm{d}}
\newcommand{\R}{\mathbb{R}}
\newcommand{\p}{\mathbf{P}}
\newcommand{\E}{\mathbf{E}}
\newcommand{\conv}{\mathrm{conv}}
\newcommand{\Vol}{\mathrm{Vol}}

\newcommand{\bN}{\mathbf{N}}
\newcommand{\ind}{\mathds{1}}

\theoremstyle{plain}
\newtheorem{theorem}{Theorem}
\newtheorem{lemma}[theorem]{Lemma}

\newtheorem{corollary}[theorem]{Corollary}

\theoremstyle{definition}

\theoremstyle{remark}
\newtheorem{remark}{Remark}

\usepackage[nobysame,alphabetic,initials]{amsrefs}

\usepackage{xcolor}

\title{Higher moments of intrinsic volumes of random beta-prime polytopes}

\author[F.~Fodor]{Ferenc Fodor$^{\orcidlink{0000-0001-9747-1981}}$}
\address{Bolyai Institute,  University of Szeged, Aradi v\'ertan\'uk tere 1, H-6720 Szeged, Hungary}
\email{fodorf@math.u-szeged.hu}

\author[B.~Gr\"unfelder]{Bal\'azs Gr\"unfelder$^{\orcidlink{0009-0008-2447-3378}}$}
\address{Bolyai Institute, University of Szeged, Aradi v\'ertan\'uk tere 1, H-6720 Szeged, Hungary}
\email{grunfibalu@server.math.u-szeged.hu}

\author[P.~Kevei]{P\'eter Kevei$^{\orcidlink{0000-0001-5872-9644}}$}
\address{Bolyai Institute,  University of Szeged, Aradi v\'ertan\'uk tere 1, H-6720 Szeged, Hungary}
\email{kevei@math.u-szeged.hu}



\begin{document}

\begin{abstract}
We consider beta-prime polytopes, i.e., the convex hulls of iid random points chosen according to beta-prime distributions in $\R^d$. 
After suitable scaling, beta-prime polytopes converge in distribution to the convex hulls of Poisson point processes with power-law intensity functions.  We prove moment convergence for the volume and all intrinsic volumes.
    
Beta-prime polytopes are the push-forwards of spherical random polytopes on the upper open half-sphere of the unit sphere $S^d\subset \R^{d+1}$. 
We prove convergence of moments of the spherical volume difference of the half-sphere and the spherical random polytopes.
\end{abstract}

\maketitle

\section{Introduction and results}

In this paper, the following two probability models are considered.
First, let $\Pi_{d,\alpha}$ be a Poisson point process on $\R^d\setminus\{0\}$ 
with the intensity measure 
$\nu_{d, \alpha}(\dd x) = f_{d,\alpha}(x) \dd x$ with respect to the Lebesgue measure, having 
a power-law density function 
\begin{equation} \label{eq:f-def}
f_{d,\alpha}(x) = |x|^{-d-\alpha}, \quad x\in\R^d\setminus\{0\}, 
\end{equation}
for some $\alpha>0$, where $|\cdot|$ denotes the Euclidean norm. For general information on Poisson point processes, we refer to the book \cite{LastPenrose} by Last and Penrose.

Second, let $\tilde\xi_1,\tilde\xi_2,\ldots$ be independent and identically distributed (iid) random points chosen according to the beta-prime distribution for a fixed $\beta>\frac d2$, i.e., their probability density functions with respect to the Lebesgue measure are
\begin{equation} \label{eq:tildef-def}    
\tilde f_{d,\beta}(x)=\tilde c_{d,\beta}(1+|x|^2)^{-\beta}, \quad x\in\R^d, \,\quad \text{ with }\,
\tilde c_{d,\beta}=\frac{\Gamma(\beta)}{\pi^{\frac d2}\Gamma(\beta-\frac d2)}, 
\end{equation}
where $\Gamma(\cdot)$ is Euler's gamma function.
With the notation $\conv{(P)}$ for the convex hull of a set of points $P\subset\R^d$, $\widetilde K_n^\beta=\conv(\{\tilde\xi_1,\ldots,\tilde\xi_n\})$ is called a random beta-prime polytope.
If it is necessary to indicate the dimension $d$ of the space from which the random points are selected, we write $\widetilde K_{n,d}^\beta=\widetilde K_n^\beta$. 
Regarding the terminology used in the paper, we note that 
in stochastic geometry the distribution with density $\tilde f_{d,\beta}$ is commonly called a beta-prime distribution. However, in
other areas of probability, beta-prime distribution or 
beta distribution of the second kind refers to a two-parameter
family of distributions.

The two models introduced above are related in the following way.
Let $\beta=\frac{d+\alpha}2$% for $\alpha>0$
, and rescale $\tilde\xi_1,\ldots,\tilde\xi_n$ by $n^{-\frac1\alpha}$. 
Then, by classical point process theory, see, e.g.~Proposition 3.21 in \cite{Resnick87}, 
the point process formed by $n^{-\frac{1}{\alpha}}\tilde\xi_1$, $\ldots$, $n^{-\frac{1}{\alpha}}\tilde\xi_n$ converges weakly to $\Pi_{d,\alpha}$, as $n\to\infty$, in the space of locally finite integer measures on $\R^d\setminus\{0\}$ endowed with the vague topology.
Furthermore, the corresponding convex hulls also converge; that is,
\begin{equation} \label{eq:convhull-conv}
n^{-\frac{1}{\alpha}}\widetilde K_n^\beta \stackrel{\mathcal{D}}{\longrightarrow}
\conv(\Pi_{d, \alpha}) \quad \text{as } \, n \to \infty,
\end{equation}
in the space of compact convex sets in $\R^d$ endowed with the Hausdorff metric; see, e.g.~\cite{DavisMulrowResnick}*{Theorem 3.1} and \cite{BroziusdeHaan}*{Theorem 2.1}.
Here and later on $\stackrel{\mathcal{D}}{\rightarrow}$ stands for convergence in distribution.

The volume (Lebesgue measure) of a $d$-dimensional convex body is denoted by $\Vol_d$, while for the $s$-dimensional volume ($s\leq d$), $\Vol_s$ will be used.
For a $d$-dimensional convex body $K \subset \R^d$, the intrinsic volumes $V_0(K), \dots, V_d(K)$ are defined as the coefficients in the Steiner formula
\[
\Vol_d(K + \varrho B^d) = \sum_{s=0}^{d} \kappa_{d-s} \, V_s(K) \, \varrho^{d-s}, \quad \varrho \ge 0,
\]
where $K+\varrho B^d$ is the Minkowski sum of $K$ and the $d$-dimensional ball of radius $\varrho$, and $\kappa_d$ denotes the volume of the $d$-dimensional unit ball.
Some intrinsic volumes have a direct geometric meaning, i.e., $V_d(K) = \Vol_d(K)$, $V_0(K)$ is the Euler characteristic, equal to $1$ for any convex body $K$, $V_1(K)$ is a constant multiple of the mean width, and $V_{d-1}(K)$ is half of the surface area.
The intrinsic volumes of $K$ satisfy Kubota's formula
\begin{equation}\label{eq:Kubota}
V_s(K)=C_{d,s} \int_{G(d,s)}\Vol_s (K|A) \nu_s(\dd A),
\end{equation}
where $G(d,s)$ denotes the Grassmannian of the $s$-dimensional linear subspaces of $\R^{d}$, $\nu_s$ is the unique rotation invariant probability measure on $G(d,s)$, $K|A$ denotes the orthogonal projection of $K$ onto an $s$-dimensional linear subspace $A$, and $C_{d,s}$ is a constant depending only on $d$ and $s$. For more information on convex bodies and intrinsic volumes, we refer to the book \cite{Sch14} by Schneider.

Asymptotic formulas for the expectation of intrinsic volumes 
of beta-prime polytopes were first established by Affentranger \cite{A91}.
More recently, Kabluchko et al.~proved that 
the $f$-vector of $n^{-\frac{1}{\alpha}} \widetilde K_n^\beta$
converges in distribution and in moments to the $f$-vector of $\conv(\Pi_{d,\alpha})$; see \cite{KMTT19}*{Theorems 2.3 and 2.4}.
Kabluchko et 
al.~in \cite{KTZ20} provided the expectation of the complete $f$-vector.
Our main result is the convergence of moments of the intrinsic volumes of the scaled beta-prime polytope.

\begin{theorem}\label{thm:moments}
Let $\alpha,\beta>0$ such that $\beta=\frac{d+\alpha}{2}$, and let $m <\alpha$. Then for any $s = 1,\ldots, d$    
\[
\lim_{n\to\infty} \E\big[ \big(
V_s(n^{-\frac{1}{\alpha}} \widetilde K_n^\beta)\big)^m\big]=
\E[ V_s(\conv (\Pi_{d,\alpha}))^m ].
\]
In particular, for $s = d$
\[
\lim_{n\to\infty} \E\big[ \big(
\Vol_d(n^{-\frac{1}{\alpha}} \widetilde K_n^\beta)\big)^m\big]=
\E[ \Vol_d(\conv (\Pi_{d,\alpha}))^m ].
\]
\end{theorem}

We note that lower and upper bounds of the same order of magnitude for the variance of the intrinsic volumes can also be proved by mainly geometric arguments; the proof of the lower bounds can be found in a previous version of this paper on arXiv \cite{FG26}.

\section{Random points on a half-sphere} \label{sect:halfsphere}

One of the motivations behind the above results is the study of convex hulls of random points on a half-sphere. Random beta-prime polytopes are connected to random spherical polytopes, which are convex hulls of iid random points in an (open) half-sphere, via the gnomonic projection.

For $d>2$, let $S^d$ denote the unit sphere in $\R^{d+1}$.
Random polytopes in (spherically) convex bodies in $S^d$ have recently been investigated, and several Euclidean results have been transferred to the spherical setting; see, for example, Besau
 et al.~\cite{BLW18}, Besau and Th\"ale \cite{BT20}, Kabluchko and Panzo \cite{KP25}. For a short overview, we refer to Fodor and Gr\"unfelder \cite{FG25}. However, random spherical polytopes generated by $n$ iid uniform random points chosen from an open half-sphere exhibit an essentially different behavior; the number of vertices of the convex hull  converges in distribution to a finite constant. This is due to the fact that the boundary of the half-sphere has zero geodesic curvature and thus the model has no natural analogue in the Euclidean setting. This type of behavior was first pointed out 
 by Fodor et al.~\cite{FKV14} for spindle convexity, and was generalized by Marynych and Molchanov \cite{MarMol}.
Uniform random spherical polytopes in open half-spheres were investigated by B\'ar\'any et al.~in \cite{BHRS17}.
More recently, Besau et al.~\cite{B24} investigated random spherical polytopes in wedges of $S^{d}$. 

Since the positive hull of a spherical polytope is a polyhedral cone, spherical random polytopes can also be considered in the context of random cones.  There are several models of random conical tessellations of $\R^d$ that produce, by intersection with the sphere, random spherical polytopes. The investigation of such models was started by Cover and Efron \cite{CE67}. For a detailed discussion of the results on random conical tessellations, random polyhedral cones, and their connections to random polytopes in half-spheres, we refer to the book \cite{S22} by Schneider.

Let $S^d_+=\{(x_1,\ldots,x_{d+1})\in S^d:x_{d+1}>0\}$
denote the upper open half-sphere, and  
consider the rotationally symmetric spherical probability distributions around the pole $N=(0,\ldots,0,1)$ with the probability density function $p_{d,\alpha}$: 
\[p_{d,\alpha}(x)=\hat c_{d,\alpha}\cdot x_{d+1}^{\alpha-1}, \quad x=(x_1,\ldots,x_{d+1})\in S^d_+,\quad \alpha>0, \]
where $\hat c_{d,\alpha}>0$ is a suitable normalization constant depending on $d$ and $\alpha$.
In the special case $\alpha=1$, $p_{d,1}$ is the density function of the uniform distribution in $S^d_+$.
Let $\widehat \xi_1, \widehat \xi_2, \ldots$ be iid random points in $S^d_+$ chosen according to the density $p_{d,\alpha}$. 
Let $\widehat K_n^\alpha$ denote the spherical convex hull of $\widehat \xi_1, \ldots, \widehat \xi_n$, i.e., the smallest spherically convex set that contains $\widehat \xi_1, \ldots, \widehat \xi_n$.
Consider the map $g:S^d_+\to\R^d$, which we call the gnomonic projection, defined as 
    \[g(x)=\Big(\frac{x_1}{x_{d+1}},\ldots,\frac{x_d}{x_{d+1}}\Big),\]
i.e., the radial projection from the origin to the tangent hyperplane at the North-pole $N$, after which the tangent hyperplane is identified with $\R^d$.
It is known that $g$ is bijective, takes spherical convex sets to convex sets in $\R^d$, and the push-forward of the spherical Lebesgue measure $\Vol_{S^d}$ under $g$ is the weighted volume measure in $\R^d$ with the radially symmetric density $\psi(x)=(1+|x|^2)^{-\frac{d+1}{2}}$. 
Moreover, the push-forward of the probability density function $p_{d,\alpha}$ is $\tilde f_{d,\beta}$ with 
$\beta=\frac{d+\alpha}2$, 
see \cite{BesauWerner}*{Proposition 4.2}.
Thus, the gnomonic image of a random spherical polytope 
is a beta-prime polytope. However, its volume is measured with respect to the density $\psi$.

Since the gnomonic projection preserves the facial structure of a convex polytope, the results of \cite{KTZ20} mentioned above yield the limit of the expected $f$-vector in the open half-sphere. Furthermore, as $n \to \infty$
\begin{equation}\label{eq:distr-conv}
    n^\frac{1}{\alpha}\Big(\frac{\omega_{d+1}}2-\Vol_{S^d}(\widehat K_n^\alpha)\Big) 
    \stackrel{\mathcal{D}}{\longrightarrow}
    \int_{(\conv(\Pi_{d,\alpha}))^c} |y|^{-d-1} \,\dd y,
\end{equation}
where $\Vol_{S^d}$ is the spherical volume (Lebesgue measure), $\omega_{d+1}=\Vol_{S^d}(S^d)$, and $X^c=\R^d\setminus X$ denotes the complement of $X$ with respect to $\R^d$.
For $\alpha = 1$ this is exactly \cite{KMTT19}*{Theorem~2.6},
but the same proof works for general $\alpha$.
We prove moment convergence for any $m>0$.

\begin{theorem}\label{thm:spherical}
Let $\alpha>0$. Then, for any $m>0$,
    \[
    \lim_{n\to\infty}\E\Big[\Big(n^\frac{1}{\alpha}\Big(\frac{\omega_{d+1}}2-\Vol_{S^d}(\widehat K_n^\alpha)\Big)\Big)^m\Big] =
    \E\bigg[\bigg(\int_{(\conv(\Pi_{d,\alpha}))^c} |y|^{-d-1} \,\dd y\bigg)^m\bigg].
    \]
\end{theorem}

The remainder of the paper is structured as follows.
In Section~\ref{sec:Poisson}, we prove two lemmas regarding Poisson point processes with power-law intensity functions.
The first, Lemma~\ref{lemma:Poisson-d} provides tail estimates of the volume of the convex hull, 
while Lemma~\ref{lemma:PPP-sphericalvol} proves tail estimates for the right-hand side of \eqref{eq:distr-conv}. These results are of independent interest.
In Section~\ref{sec:IID}, we prove the iid counterparts of these lemmas, which imply 
the uniform integrability of the relevant variables. In Lemma \ref{lemma:intrinsic-unifint},
we extend the uniform integrability from the volume to any intrinsic volume using
Kubota's formula. As a consequence of the uniform integrability results, the proofs 
of Theorems \ref{thm:moments} and \ref{thm:spherical} follow.

Finally, we note that the results can be extended to more general regularly varying random vectors.
However, to keep the presentation simple, we focus only on the geometrically relevant 
case.

\section{Tail asymptotics for the Poisson point process}
\label{sec:Poisson}

Let $\Pi_{d,\alpha}$ be a Poisson point process on $\R^d$ with intensity measure $\nu_{d, \alpha}$ with density $f_{d,\alpha}$
in \eqref{eq:f-def}. 
In \cite{KMTT19}*{Theorem~2.12}, explicit formulas for the expectation
of the $T$-functional of $\Pi_{d,\alpha}$ are provided, which include, for instance, the volume and the number faces. 
Here, we need the tail asymptotics
of the relevant quantities of $\Pi_{d,\alpha}$.
We first determine the order of the tail 
of $\Vol_d(\conv(\Pi_{d,\alpha}))$.
In what follows, $C, c > 0$ stand for suitable constants and can change from 
line to line. 

\begin{lemma} \label{lemma:Poisson-d}
There exist $0 < c < C$ such that for any $x > 0$,
\[
c \frac{(\log x)^{d-1}}{x^\alpha} \leq 
\p ( \Vol_d(\conv(\Pi_{d,\alpha})) > x ) \leq 
C \frac{(\log x)^{d-1}}{x^\alpha}.
\]
\end{lemma}

\begin{proof}
Let $(\xi_i)$ denote the points of $\Pi_{d,\alpha}$ 
and put $X = \Vol_d(\conv(\Pi_{d,\alpha}))$.

\emph{Lower bound.}
Let $e_1,\ldots, e_d$ be the standard orthonormal basis in $\R^d$. 
For a vector $z\in\R^d\setminus\{0\}$ and $\varphi\in[0,\frac\pi2]$, let $C(z,\varphi)$ denote the closed circular cone with axis $\{tz:t\geq0\}$ and angle $\varphi$.
Let $\theta\in(0,\frac\pi5)$ and $S_k^\theta=C(e_k,\theta)$, and 
$U_k=\max\{|\xi_i|:\xi_i\in S_k^\theta\}$ 
for $k=1,\ldots,d$.
Since $0 \in \conv(\Pi_{d,\alpha})$ almost surely, 
\[ 
X = \Vol_d(\conv(\Pi_{d,\alpha})) 
\geq c_{\theta,d}\, U_1\ldots U_d,
\] 
for a constant $c_{\theta,d}$ that depends only on $\theta$ and $d$.

The variables $U_1,\ldots,U_d$ are iid and
\[
\begin{split}
\p ( U_1 > r) &  = 1  - 
e^{-\nu_{d, \alpha}( S_1^\theta \cap (r B^d)^c )} \sim 
\nu_{d, \alpha}( S_1^\theta \cap (r B^d)^c ) \\
& = \gamma(\theta,d) \int_r^\infty u^{-\alpha-1} \, \dd u 
= \frac{\gamma(\theta,d)}{\alpha} r^{-\alpha}
\quad \text{as } \, r \to \infty,
\end{split}
\]
where $\gamma(\theta,d)$ is the spherical Lebesgue measure of $S_1^\theta \cap S^{d-1}$ and $\sim$ denotes the asymptotic equality of functions.
Thus, by Lemma 4.1 (4) in \cite{JessenMikosch}, for some $c > 0$,
as $x \to \infty$
\[
\p (X > x ) \geq \p \Big( U_1 \ldots U_d > \frac {x}{c_{\theta, d}}\Big) \sim 
c \frac{(\log x)^{d-1}}{x^{\alpha}}.
\]

\emph{Upper bound.}
Let $\mu_1$ denote the Poisson point process $\{ |\xi_i|: i\geq 1 \}$ on
$(0,\infty)$. The intensity measure $\eta$ of $\mu_1$ is given by 
\[
\overline \eta (r) := 
\eta((r,\infty)) = \nu_{d, \alpha}( (rB^d)^c ) = \frac{\omega_d}{\alpha} r^{-\alpha}, 
\]
where $\omega_d $ is the surface area of the unit ball $B^d$.
Let $V_1 > V_2 > \ldots$
denote the points of $\mu_1$ in decreasing order. Then 
\begin{equation} \label{eq:Poi-d-upper}
X \leq 2^d V_1 \ldots V_d.
\end{equation}
Let $\bN$ denote the set of point measures on $(0,\infty)$. 
Fix $r > 0$.
Define $h: (0,\infty)^k \times \bN \to (0,\infty)$ as
\begin{equation*} %\label{eq:def-f}
\begin{split}
h(x_1, \ldots, x_k, \mu ) = 
\ind(x_1 > \ldots > x_k, x_1 \ldots x_k > r,
x_i \in \mu, \mu([x_k, \infty)) = k ),
\end{split}
\end{equation*}
that is, $h$ is the indicator of the event that $x_1 > \ldots > x_k$
are the $k$ largest points of $\mu$, and their product 
is larger than $r$. Then
\[
\int_{(0,\infty)^k} h(x_1, \ldots, x_k, \mu_1) 
\mu_1^{(k)} (\dd x_1, \ldots, \dd x_k) = 
\ind ( V_1 \ldots V_k > r),
\]
where $\mu_1^{(k)}$ stands for the $k$th factorial measure 
of $\mu_1$. By the multivariate Mecke formula 
(\cite{LastPenrose}*{Theorem 4.4})
\begin{equation} \label{eq:mecke-aux1}
\begin{split}
& \p ( V_1 \ldots V_k > r) \\
& = 
\E \bigg[ \int_{(0,\infty)^k} h(x_1, \ldots, x_k, \mu_1) \mu_1^{(k)}
\dd (x_1, \ldots, x_k) \bigg] \\
& = \int \ldots \int_{(0,\infty)^k}
\E [ h(x_1, \ldots, x_k, \mu_1 + \delta_{x_1} + \ldots 
+ \delta_{x_k}) ] \eta(\dd x_1) \ldots \eta(\dd x_k).
\end{split}
\end{equation}
We have 
\[
\begin{split}
& \E [ h(x_1, \ldots, x_k, \mu_1 + \delta_{x_1} + \ldots 
+ \delta_{x_k}) ] \\
& = \ind( x_1 > \ldots > x_k, x_1 \ldots x_k > r) 
\p ( \mu_1 ([x_k, \infty) )= 0) \\
& = \ind( x_1 > \ldots > x_k, x_1 \ldots x_k > r) 
e^{-\overline \eta(x_k)}.
\end{split}
\]
Substituting back into \eqref{eq:mecke-aux1}
\begin{equation}  \label{eq:mecke}
\begin{split}
& \p ( V_1 \ldots V_k > r) \\
& = 
\int \ldots \int_{(0,\infty)^k} 
\ind(x_1 > \ldots > x_k, x_1 \ldots x_k > r)
e^{-\frac{\omega_d}{\alpha} x_k^{-\alpha}} \omega_d^k 
\prod_{i=1}^k x_i^{-\alpha -1} \dd x_i.
\end{split}
\end{equation}
To ease notation, put $x_{1:i} = x_1 \ldots x_i$.
We show that for $k= 1, 2, \ldots$
\begin{equation} \label{eq:V1d-ind}
\limsup_{r \to \infty} \frac{r^\alpha}{(\log r)^{k-1}} 
\p ( V_1 \ldots V_k > r) \leq \frac{\omega_d^{k}}{\alpha   (k-1)!}.
\end{equation}
Taking this for granted, the result follows by \eqref{eq:Poi-d-upper}.
From  \eqref{eq:mecke} with $k= 1$, we obtain
\[
\p ( V_1 > r ) \sim \frac{\omega_d}{\alpha} r^{-\alpha}
\quad \text{as } \, r \to \infty,
\]
thus \eqref{eq:V1d-ind} holds for $k=1$. Let $k \geq 2$ and assume
that \eqref{eq:V1d-ind} holds for $k-1$. 
Then
\[
\begin{split}
\p (V_1 \ldots V_k > r) &  
= \p( V_1 \ldots  V_k > r \geq V_1 \ldots V_{k-1} ) + \p ( V_1 \ldots V_{k-1} > r) \\
& = \p( V_1 \ldots  V_k > r \geq V_1 \ldots V_{k-1} ) + O\bigg(\frac{(\log r)^{k-2}}{r^\alpha}\bigg).
\end{split}
\]
For the first term above, we have
\[
\begin{split}
& \p( V_1 \ldots  V_k > r \geq V_1 \ldots V_{k-1} ) \\
&   = \int \ldots \int_{(0,\infty)^k} 
\ind(x_1 > \ldots > x_k, x_{1:k} > r \geq x_{1:(k-1)})
e^{- \frac{\omega_d}{\alpha} x_k^{-\alpha}} \omega_d^k 
\prod_{i=1}^k x_i^{-\alpha -1} \dd x_i  \\
& \leq \int \ldots \int_{[1,\infty)^k} 
\ind(r > x_1 > \ldots > x_k, x_{1:k} > r > x_{1:(k-1)})
\omega_d^k x_{1:k}^{-\alpha -1} \dd x_1 \ldots \dd x_k \\
& \leq \int \ldots \int_{[1,\infty)^{k-1}} 
\ind(r > x_1 > \ldots > x_{k-1})
\omega_d^{k} \alpha^{-1} x_{1:(k-1)}^{-1} r^{-\alpha} 
\dd x_1 \ldots \dd x_{k-1} \\
& \leq \frac{\omega_d^{k}}{\alpha} r^{-\alpha}
\int_{1}^r x_1^{-1} \dd x_1
\prod_{j=2}^{k-1} \int_{1}^{x_{j-1}} x_j^{-1} \dd x_j \\
& = \frac{\omega_d^{k}}{\alpha} r^{-\alpha}
\frac{(\log r)^{k-1}}{(k-1)!}.
\end{split}
\]
Thus \eqref{eq:V1d-ind} holds for $k$, and the result follows.
\end{proof}

\begin{remark}
We conjecture that the stronger statement 
$\p ( X > x ) \sim c \frac{(\log x)^{d-1}}{x^\alpha}$ 
as $x \to \infty$ also holds.
\end{remark}

Next, we obtain the tail asymptotics of the integral 
corresponding to the missed spherical volume in
Theorem \ref{thm:spherical}.

\begin{lemma} \label{lemma:PPP-sphericalvol}
Let 
\begin{equation*} %\label{eq:Y-def}
Y = \int_{(\conv(\Pi_{d,\alpha}))^c} |y|^{-d - 1} \dd y.    
\end{equation*}
Then there exist $c > 0$ and $C > 0$ such that for all $y > 0$
\[
\p ( Y > y ) \leq C e^{-c y^{\alpha}}.
\]
In particular, $\E [Y^m] < \infty$ for any $m > 0$.
\end{lemma}

\begin{proof}
Let $U = \mathrm{dist}(0, (\conv(\Pi_{d,\alpha}))^c)$, where 
$\mathrm{dist}(\cdot,\cdot)$ stands for the distance. Then
\begin{equation} \label{eq:Y-bound}
\begin{split}
Y & \leq \int_{(UB^d)^c} | y|^{-d - 1} \dd y 
%\\ & 
= \omega_d \int_U^\infty r^{-2} \dd r 
=  \omega_d U^{-1}.
\end{split}
\end{equation}
Furthermore,
\begin{equation} \label{eq:Poi-conv-auxU}
\begin{split}
\p (U \leq r )  & = \p ( rB^d \not \subset \conv(\Pi_{d,\alpha})) ) \\
& \leq 2 d  \p ( \Pi_{d, \alpha} ( C(e_1, \theta) 
\cap (c_\theta rB^d)^c)  = 0 ) \\
& = 2 d e^{- \nu_{\alpha,d}(S_1^\theta \cap (c_\theta rB^d)^c)} \\
& = 2 d \exp \Big\{ - \frac{\gamma(\theta, d)}{\alpha} 
(c_\theta r)^{-\alpha} \Big\},
\end{split}
\end{equation}
where we used that for $\theta \in (0,\frac{\pi}{5})$
there exists $c_\theta > 1$ such that
\begin{equation} \label{eq:conv-bound-aux}
\begin{split}
\{ rB^d \not \subset \conv(\Pi_{d,\alpha}) \} 
\subset 
& \cup_{i=1}^d \big( 
\{ \Pi_{d,\alpha} ( C(e_i, \theta) \cap (c_\theta rB^d)^c)  = 0 \} \\
&  
\cup \{ \Pi_{d,\alpha} ( C(-e_i, \theta) \cap (c_\theta rB^d)^c)  = 0 \} \big).
\end{split}
\end{equation}
Thus, combining \eqref{eq:Poi-conv-auxU} and \eqref{eq:Y-bound}, 
the statement follows.
\end{proof}

\section{Proofs of the main results} \label{sec:IID}

We first need a technical lemma on uniform integrability.

\begin{lemma} \label{lemma:iid-d-tail}
Let $Y, Y_1, \ldots$ be iid random variables such that 
$\p ( Y > x) \sim a x^{-\alpha}$, as $x \to \infty$,
for some $a > 0$.
Let $Y_{1,n} \geq \ldots \geq Y_{n,n}$ denote the order statistics
of $Y_1, \ldots, Y_n$. Then for any $k \geq2$, there exists $c > 0$
such that for all $x > 0$
\[
\sup_{n \geq 1} \p \big( Y_{1,n} \ldots Y_{k,n} > 
n^{\frac k\alpha} x \big) 
\leq c \frac{(\log x)^{k-1}}{x^\alpha}.
\]
\end{lemma}

\begin{proof}
We follow the proof of the Poisson case.
Put $a_n = n^{\frac1\alpha}$.
We prove by induction that for some $c_k> 0$
\begin{equation}\label{eq:iid-ind}
\limsup_{r \to \infty} \, \sup_n \frac{r^\alpha}{(\log r)^{k-1}} 
\p ( Y_{1,n} \ldots Y_{k,n} > a_n^k r ) \leq c_k.
\end{equation}
For $k=1$ \eqref{eq:iid-ind} holds, and assume it for $k-1$ with $k \geq 2$.
Then,
\[
\begin{split}
\p ( Y_{1,n} \ldots Y_{k,n} > a_n^k r ) & \leq  
\p ( Y_{1,n} \ldots Y_{k,n} > a_n^k r, Y_{1,n} \ldots Y_{k-1,n} \leq  a_n^{k-1} r ) \\ 
& \quad +
\p ( Y_{1,n} \ldots Y_{k-1,n} > a_n^{k-1} r ),
\end{split}
\]
where the last term is $O((\log r)^{k-2} r^{-\alpha})$ by the induction 
hypothesis. To ease notation, put $y_{1:i} = y_1 \ldots y_i$, $i = 1,2,\ldots$.
For the first term, note that $\{ Y_{1,n} \ldots Y_{k,n} > a_n^k r \}$
and $\{ Y_{1,n} \ldots Y_{k-1,n} \leq  a_n^{k-1} r \}$ together imply 
that $Y_{k,n} \geq a_n$ and $Y_{1,n} \leq a_n r$. Therefore,
putting $F(x) = \p ( Y \leq x)$,
we have
\begin{equation} \label{eq:iid-d-tail-aux1}
\begin{split}
& 
\p ( Y_{1,n} \ldots Y_{k,n} > a_n^k r, 
Y_{1,n} \ldots Y_{k-1,n} \leq  a_n^{k-1} r ) \\
& \leq n^k \, \p ( Y_1 \ldots Y_k > a_n^k r, Y_1 \ldots Y_{k-1} \leq a_n^{k-1} r,
Y_1 > Y_2 > \ldots > Y_k ) \\
& \leq n^k  
\int_{(a_n r^{\frac1k}, a_n r]} \dd F(y_1) 
\int_{(a_n, y_1]} \dd F(y_2) 
\ldots \\ & \qquad 
\int_{(a_n, y_{k-2}]}
\p \Big( Y_k > \frac{a_n^k r}{y_{1:(k-1)}} \Big) 
\dd F(y_{k-1}).
\end{split}
\end{equation}

For $x > a_n$ and $i = 0,1,2, \ldots$, integration by parts gives
\[
\begin{split}
& \int_{(a_n, x]} y^\alpha (\log y)^{i} \dd F(y) \\
& = 
- \big[ y^\alpha (\log y)^i \overline F(y) \big]_{a_n}^x
+ \int_{a_n}^x \overline F(y) \big( \alpha y^{\alpha -1 } (\log y)^{i} 
+ y^{\alpha-1} i (\log y)^{i-1} \big) \dd y \\
& \leq c  \int_{a_n}^x \frac{(\log y)^{i}}{y} \dd y 
\leq c  \frac{(\log x)^{i+1}}{i+1}.
\end{split}
\]
Therefore, for the inner integral in \eqref{eq:iid-d-tail-aux1},
\[
\begin{split}
\int_{(a_n, y_{k-2}]}
\p \bigg( Y_k > \frac{a_n^k r}{y_{1:(k-1)}} \bigg) 
\dd F(y_{k-1}) 
& \sim a \int_{(a_n, y_{k-2}]} 
\bigg( \frac{a_n^k r}{y_{1:(k-1)}} \bigg)^{-\alpha} 
\dd F(y_{k-1}) \\
& \leq c \bigg( \frac{a_n^k r}{y_{1:(k-2)}} \bigg)^{-\alpha}
 \log y_{k-2} .
\end{split}
\]
Similarly,
\[
\begin{split}
& \int_{(a_n, y_{k-3}]}
\bigg( \frac{a_n^k r}{y_{1:(k-2)}} \bigg)^{-\alpha} 
\log y_{k-2}  \, \dd F(y_{k-2}) 
\leq c 
\bigg( \frac{a_n^k r}{y_{1:(k-3)}} \bigg)^{-\alpha} 
(\log y_{k-3})^2.
\end{split}
\]
Continuing, we arrive at
\[
\begin{split}
& \p ( Y_{1,n} \ldots Y_{k,n} > a_n^k r, 
Y_{1,n} \ldots Y_{k-1,n} \leq  a_n^{k-1} r )  \\
& \leq c n^k
\int_{(a_n r^{\frac1k}, a_n r]} 
\bigg( \frac{a_n^k r}{y_1} \bigg)^{-\alpha} 
(\log y_{1})^{k-2} \dd F(y_1) \\
& \leq 
c r^{-\alpha} 
(\log r)^{k-1},
\end{split}
\]
showing that \eqref{eq:iid-ind} holds for $k$. The proof is complete.
\end{proof}

\begin{remark}
The upper bound in Lemma \ref{lemma:iid-d-tail} is, in fact, precise.
By the point process convergence, we have
\[
n^{-\frac{1}{\alpha}} \left( Y_{1,n}, \ldots , Y_{k, n} \right)
\stackrel{\mathcal{D}}{\longrightarrow} 
\left( \Gamma_1^{-\frac{1}{\alpha}}, \ldots, \Gamma_k^{-\frac{1}{\alpha}}
\right),
\]
where $\Gamma_i = \varepsilon_1 + \ldots + \varepsilon_i$, $i = 1,2,\ldots, k$,
and $\varepsilon_1, \varepsilon_2, \ldots$ are iid standard exponential 
random variables.
Thus
\[
\lim_{n \to \infty} \p ( n^{-\frac{k}{\alpha}} Y_{1,n} \ldots Y_{k,n} > x) 
= \p ( (\Gamma_1 \ldots \Gamma_k)^{-1} > x^{\alpha} ).
\]
The representation
$\Gamma_k^{-1} ( \Gamma_1, \ldots, \Gamma_{k-1}) 
\stackrel{\mathcal{D}}{=} ( U_{1,k-1}, \ldots, U_{k-1,k-1} )$,
where $\Gamma_k$ is independent of the left-hand side, and 
the right-hand side is the order statistics of $k-1$ independent standard 
uniform random variables, implies that as $x \to \infty$
\[
\p ( (\Gamma_1 \ldots \Gamma_k)^{-1} > x^\alpha)
\sim \frac{\alpha^{k-1} (\log y)^{k-1}}{x^{\alpha} (k-1)!}.
\]
\end{remark}

Let $\tilde \xi, \tilde \xi_1, \ldots $ iid random points 
in $\R^d$ with density functions $\tilde f_{d,\beta}$ in \eqref{eq:tildef-def}.
Then, for the distribution of the length $|\tilde \xi|$, we have
\begin{equation*} %\label{eq:xi-asy}
\p ( |\tilde \xi | > x) = \omega_{d} \tilde c_{d,\beta}
\int_x^\infty r^{d-1} (1 + r^2)^{-\beta} \dd r \sim 
\frac{\omega_{d} \tilde c_{d,\beta}}{2\beta -d } x^{-(2\beta - d)}
\quad \text{as } x \to \infty.
\end{equation*}
In particular, the conditions of Lemma \ref{lemma:iid-d-tail} are
satisfied.
Let $|\tilde \xi_{1,n}| \geq \ldots 
\geq |\tilde \xi_{n,n}|$ denote the order statistics of 
$|\tilde \xi_1|, \ldots, |\tilde \xi_n |$.
Then 
\[
\Vol_n(\widetilde K_n^{\beta}) \leq 2^d |\tilde \xi_{1,n} | \ldots 
|\tilde \xi_{d,n} |.
\]
Therefore, as an immediate consequence of Lemma \ref{lemma:iid-d-tail}, 
we have the following.

\begin{corollary} \label{cor:X_n-unifint}
Let $\alpha = 2 \beta -d$. There exists $c > 0$ such that for any $x > 0$
\[
\sup_n \p ( \Vol_d(n^{-\frac{1}{\alpha}} \tilde K_n^{\beta})) >  x) 
\leq c \frac{(\log x)^{d-1}}{x^\alpha}.
\]
In particular, the sequence 
$((\Vol_d(n^{-\frac{1}{\alpha}}\tilde K_n^{\beta}))^m)_{n}$
is uniformly integrable for $m < \alpha$.
\end{corollary}

Using Kubota's formula, the uniform integrability of the 
intrinsic volumes also follows.

\begin{lemma} \label{lemma:intrinsic-unifint}
For any $m < \alpha$ and $s = 1,\ldots, d$, the sequence 
$\left( V_s(n^{-\frac{1}{\alpha}} \widetilde K_n^\beta)^m \right)_{n \geq 1}$
is uniformly integrable.
\end{lemma}

\begin{proof}
By Kubota's formula \eqref{eq:Kubota}, and the rotation invariance of the beta-prime distribution, we obtain that
\[
\E[V_s(\widetilde K_n^\beta)]=\E\Bigl [C_{d,s}\int_{G(d,s)}\Vol_s(\widetilde K_n^\beta|A)\nu_s(\dd A)\Bigr]
=C_{d,s} \E[\Vol_s(\widetilde K_n^\beta|A_0)],
\]
where $C_{d,s}$ is the constant from Kubota's formula, and $A_0$ is an $s$-dimensional linear subspace of $\R^d$. Let $A_0$ be identified with $\R^s$.

The probability density function of $\tilde\xi|A_0$ is $\tilde f_{s,\beta-\frac{d-s}2}$ (see \cite{KTT19}*{Lemma 4.4}).
Thus, 
\[
\E\big[\Vol_s\big(\widetilde K_n^\beta|A_0\big)\big]=\E\Big[\Vol_s\Big(\widetilde K_{n,s}^{\beta-\frac{d-s}2}\Big)\Big],
\]
and 
\[
\E[V_s(\widetilde K_n^\beta)] = C_{d,s}\E\Big[\Vol_s\Big(\widetilde K_{n,s}^{\beta-\frac{d-s}2}\Big)\Big].
\]

Similarly, for the moments of $V_s(\widetilde K_n^\beta)$, applying Jensen's inequality and Fubini's theorem, we have for any $1 < m < 2 \beta - d = \alpha$
\begin{align*}
\E \big[ V_s(\widetilde K_n^\beta)^m \big]
&=\E\bigg[ \Big(C_{d,s}\int_{G(d,s)}\Vol_s(\widetilde K_n^\beta|A)\nu_s(\dd A)
\Big)^m\bigg] \nonumber \\
&\leq C_{d,s}^m
\E\bigg[\int_{G(d,s)}
\Vol_s(\widetilde K_n^\beta|A)^m \, \nu_s(\dd A)\bigg] \nonumber \\
&= C_{d,s}^m 
\E \bigg[ 
\Vol_s\Big( \widetilde K_{n,s}^{\beta-\frac{d-s}2}\Big)^m \bigg].%\label{eq:vs-var}
\end{align*}
Note the dependence on the dimension in the last formula.
Therefore,
\[
\begin{split}
\E ( V_s ( n^{-\frac{1}{\alpha}} \widetilde K_n^{\beta} )^m ) 
& = n^{-\frac{ms}{\alpha}} \E ( V_s(\widetilde K_n^{\beta})^m ) \\
& \leq C_{d,s}^m n^{-\frac{ms}{\alpha}} 
\E ( \Vol_s(\widetilde K_{n,s}^{\beta-\frac{d-s}{2}})^m ) \\
& = C_{d,s}^m  \E ( 
\Vol_s( n^{-\frac{1}{\alpha}} \widetilde K_{n,s}^{\beta-\frac{d-s}{2}})^m ).
\end{split}
\]
The statement follows from the fact that 
the sequence 
$( \Vol_s (n^{-\frac{1}{\alpha}} 
\widetilde K_{n,s}^{\beta-\frac{d-s}2})^m)_{n \geq 1}$ is uniformly integrable 
by Corollary \ref{cor:X_n-unifint}, whenever 
$m < 2 (\beta - \frac{d-s}{2}) - s = 2 \beta -d = \alpha$.
\end{proof}

\begin{proof}[Proof of Theorem \ref{thm:moments}.]
Since the volume and all intrinsic volumes are continuous mappings 
in the Hausdorff metric, the continuous 
mapping theorem and \eqref{eq:convhull-conv} imply that 
for all $s = 1,\ldots, d$
\[
V_s(n^{-\frac{1}{\alpha}} \widetilde K_n^{\beta}) \stackrel{\mathcal{D}}{\longrightarrow}
V_s(\conv(\Pi_{d, \alpha})).
\]
By Lemma \ref{lemma:intrinsic-unifint} the sequence 
$(V_s(n^{-\frac{1}{\alpha}} \widetilde K_n^{\beta}))^m$ is uniformly integrable for any $m < \alpha$, thus the statement follows from the moment convergence theorem.
\end{proof}

\begin{lemma}\label{lemma:spherical-iid}
Let 
\begin{equation*} %\label{eq:Y-n-def}
Y_n = n^{\frac{d+1}{\alpha}} 
\int_{\big(n^{-\frac1\alpha} \widetilde K_n^\beta\big)^c}
\Big(1 + \big| y n^{\frac1\alpha} \big|^2\Big)^{-\frac{d+1}{2}} \dd y.
\end{equation*}
Then, there exists $c, C > 0$ such that for all $y > 0$
\[
\p ( Y_n > y) \leq C e^{-c y^{\alpha}}.
\]
In particular, the sequence $(Y_n^m)_{n \geq 1}$ is uniformly integrable
for any $m > 0$.
\end{lemma}

\begin{proof}
As in the Poisson case, define
\[
U_n = \mathrm{dist}\big(0, \big(n^{-\frac1\alpha} \widetilde K_n\big)^c\big).
\]
Then 
\[
\begin{split}
Y_n & \leq c n^{\frac{d+1}{\alpha}} \int_{U_n}^\infty 
y^{d-1} \Big( 1 + \big(y n^{\frac1\alpha} \big)^2 \Big)^{- \frac{d+1}{2}} 
\dd y \\
& \leq c n^{\frac{d+1}{\alpha}} 
\big(  
\ind\big(U_n > n^{-\frac1\alpha}\big) + 
\ind\big( U_n \leq n^{-\frac1\alpha}\big) \big)
\int_{U_n}^\infty
y^{d-1} \Big( 1 + \big( y n^{\frac1\alpha} \big)^2 \Big)^{- \frac{d+1}{2}} 
\dd y \\
& = c n^{\frac{d+1}{\alpha}} ( I_1 + I_2).
\end{split}
\]
Simply, 
\[
\begin{split}
I_1 \leq \ind\big(U_n > n^{-\frac1\alpha}\big) \int_{U_n}^\infty 
y^{d-1} \big( y n^{\frac1\alpha} \big)^{- {d-1}} \dd y 
\leq \ind\big(U_n > n^{-\frac1\alpha}\big) c n^{-\frac{d+1}{\alpha}} U_n^{-1}.
\end{split}
\]
and 
\[
\begin{split}
I_2 \leq \ind \big(U_n \leq n^{-\frac1\alpha}\big)  
\left( \int_{0}^{n^{-\frac1\alpha}} y^{d-1} \dd y
+ cn^{-\frac{d}{\alpha}} \right)
\leq c \ind\big(U_n \leq n^{-\frac1\alpha}\big) n^{-\frac{d}{\alpha}}.
\end{split}
\]
Summarizing,
\[
Y_n \leq c \Big( U_n^{-1} \wedge n^{\frac1\alpha} \Big).
\]
Next, for $r \geq c n^{-\frac1\alpha}$, by \eqref{eq:conv-bound-aux}
\[
\begin{split}
\p ( U_n \leq r) & \leq 2d \Big( \p \Big( \xi \not \in C(e_1, \theta) \cap 
\big(c _\theta r n^{\frac1\alpha}B^d\big)^c \Big) \Big)^n \\
& \leq c \bigg( 1 - \frac{c r^{-\alpha}}{n} \bigg)^n
% \\ & 
\leq c e^{-c r^{-\alpha}},
\end{split}
\]
while for some $\delta > 0$
\[
\p \big( U_n \leq c n^{-\frac1\alpha} \big) \leq ( 1- \delta)^n.
\]
and the statement follows as in the Poisson case.
\end{proof}

As a consequence, we obtain Theorem \ref{thm:spherical}.

\begin{proof}[Proof of Theorem \ref{thm:spherical}.]
The distributional convergence \eqref{eq:distr-conv},  Lemma \ref{lemma:intrinsic-unifint},
and the moment convergence theorem imply the statement.
\end{proof}

\section*{Acknowledgments}
This research was supported by NKFIH project no.~150151. Project no.~150151 has been implemented with the support provided by the Ministry of Culture and Innovation of Hungary from the National Research, Development and Innovation Fund, financed under the ADVANCED\_24 funding scheme.

The second author was also supported by the University Research Scholarship Programme (EK\"OP) no.~EK\"OP-452-SZTE, which has been implemented with the support provided by the Ministry of Culture and Innovation and the National Research, Development and Innovation Fund.

\end{document}